\documentclass[12pt]{article}

\usepackage[top=1in,bottom=1in,left=1.5in,right=1.2in,footskip=0.4in]{geometry}
\usepackage{amssymb}
\usepackage{amsfonts}
\usepackage{amsmath}
\usepackage{amsthm}
\usepackage{mathrsfs}
\usepackage{hyperref}

\usepackage{tikz,tikz-cd,tkz-graph,enumerate}
\usepackage{calc}
\usepackage{float}

\usetikzlibrary{matrix,arrows,decorations.pathmorphing}
\usetikzlibrary{graphs}
\usetikzlibrary{graphs.standard}
\usetikzlibrary{calc}

\newtheorem{THEOREM}{Theorem}[section]
\newtheorem{LEMMA}[THEOREM]{Lemma}

\newtheorem{REMARK}[THEOREM]{Remark}
\newtheorem{PROPOSITION}[THEOREM]{Proposition}
\newtheorem{COROLLARY}[THEOREM]{Corollary}
\newtheorem{PROBLEM}[THEOREM]{Problem}
\newtheorem{CONJECTURE}[THEOREM]{Conjecture}

\newtheorem{OBSERVATION}[THEOREM]{Observation}

\newcommand{\Dc}{\mathcal{D}}
\newcommand{\K}{\mathcal{K}}
\newcommand{\Ply}{\mathcal{P}}
\newcommand{\M}{\mathcal{M}}
\newcommand{\C}{\mathcal{C}}
\newcommand{\Q}{\mathcal{Q}}
\newcommand{\INT}{\mathrm{Int}}

\title{Covering the Plane by a Sequence of Circular Disks with a Constraint}
\author{Amitava Bhattacharya, Anupam Mondal
	\footnote{Present affiliation: The Institute of Mathematical Sciences, HBNI, Chennai, India. Email: \tt{anupamm\,@\,imsc.res.in}}\\
	\small School of Mathematics\\[-0.8ex]
	\small Tata Institute of Fundamental Research\\[-0.8ex] 
	\small Mumbai, India\\
	\small\tt \{amitava, anupamm\} @\,math.tifr.res.in\\
}
 \date{} 

\begin{document}

\maketitle

\begin{abstract}
We are interested in the following problem of covering the plane by a sequence of congruent circular disks with a constraint on the distance between consecutive disks. Let $(\Dc_n)_{n \in \mathbb N}$ be a sequence of closed unit circular disks such that $\cup_{n \in \mathbb{N}} \Dc_n = \mathbb {R}^2$ with the condition that for $n \ge 2$, the center of the disk $\Dc_n$ lies in $\Dc_{n-1}$. What is a ``most economical" or an optimal way of placing $\Dc_n$ for all $n \in \mathbb{N}$? 
We answer this question in the case where no ``sharp" turn is allowed, i.e. if $C_n$ is the center of the disk $\Dc_n$, then for all $n \ge 2$, %
$\angle C_{n-1}C_nC_{n+1}$ is not very small.

We also consider a related problem. We wish to find out an optimal way to cover the plane with unit circular disks with the constraint that each disk contains the centers of at least two other disks. We find out the answer in the case when the centers of the disks form a two-dimensional lattice.
\end{abstract}

\section{Introduction}
Let $\Dc$ denote the (closed) unit circular disk $\{(x,y) \in \mathbb {R}^2 : x^2 + y^2 \le 1\}$ and $\mathcal I$ denote the square $[-1,1]\times[-1,1]$ in  $\mathbb {R}^2$. If $\mathscr F = \{\Dc_n : n \in \mathbb N\}$ is a family of unit circular disks in $\mathbb {R}^2$ (i.e. for all $n \in \mathbb N$, $\Dc_n$ is a congruent copy of $\Dc$) such that $\cup_{n \in \mathbb{N}} \Dc_n = \mathbb {R}^2$, then we say the family of disks $\mathscr F$ covers $\mathbb {R}^2$ and $\mathscr F$ is a \emph{covering} of $\mathbb {R}^2$ by unit circular disks. If $(\Dc_n)_{n \in \mathbb N}$ is a sequence of closed unit circular disks such that $\cup_{n \in \mathbb{N}} \Dc_n = \mathbb {R}^2$, then we say $(\Dc_n)_{n \in \mathbb N}$ is a \emph{sequence covering} of the plane (or $\mathbb {R}^2$) by unit circular disks. Similarly, if $X$ is a bounded subset of $\mathbb{R}^2$ and $(\Dc_1, \ldots, \Dc_N)$ is a finite sequence of closed unit circular disks such that $X \subset \cup_{n = 1}^N \Dc_n$, then we say $(\Dc_1, \ldots, \Dc_N)$ is a finite sequence covering of $X$.

For $X \subset \mathbb{R}^2$ and a positive real number $\lambda$, 
by $\lambda X$ we mean the set $\{\lambda\cdot x : x \in X\}$. We use $\INT(X)$ and $a(X)$ to denote the interior and area of set $X$ (the Lebesgue measure of $X$, whenever it exists), respectively. 

Let $N_\lambda$ denote the cardinality of the subfamily  of $\mathscr F$ consisting of only the disks that intersect  $\INT (\lambda \mathcal  I)$ and $\gamma (\lambda,\mathscr F)$ denote the ratio $\frac{N_\lambda \cdot a(\Dc)}{a(\lambda \mathcal  I)}$. We define the \emph{lower density} and \emph{density} of the covering $\mathscr F$ to be $\liminf_{\lambda \to \infty} \gamma (\lambda, \mathscr F)$ and $\lim_{\lambda \to \infty} \gamma (\lambda, \mathscr F)$, respectively, if the limits exist. The optimal lower density 
of covering the plane by unit circular disks is denoted by $\gamma (\Dc)$ and is defined as   
$$\gamma (\Dc) = \inf_{\mathscr F} \liminf_{\lambda \to \infty} \gamma (\lambda, \mathscr F).$$

We note that in the definitions of $N_\lambda$ and $\gamma (\lambda,\mathscr F)$, the square $\mathcal  I$ can be substituted by any convex bounded set in $\mathbb R^2$ which contains the origin as the set can be arbitrarily approximated by the union of a sequence of squares. 
We refer to \cite[Chapters~2--4]{PA} and \cite{ZONG2014297} for background and basic concepts.

Kershner proved in \cite{K} that $\gamma (\Dc) \ge \frac{2\pi}{\sqrt{27}}$ and this bound can be achieved when centers of the circles are arranged in a regular hexagonal lattice (see Figure~\ref{fig-kershner}).
\begin{figure}[!ht]
	\centering
	\begin{tikzpicture}
	\begin{scope}[scale=0.8, transform shape]
	\clip(-1,-4) rectangle (12,5);
	
	\foreach \a in {-2, -1, ..., 8}
	\draw (\a*1.732,0) circle (1);
	
	\foreach \a in {-2, -1, ..., 8}
	\draw [xshift=0.866cm,yshift=1.5cm] (\a*1.732,0) circle (1);
	
	\foreach \a in {-2, -1, ..., 8}
	\draw [xshift=0.866cm,yshift=-1.5cm] (\a*1.732,0) circle (1);
	
	\foreach \a in {-2, -1, ..., 8}
	\draw [yshift=3cm] (\a*1.732,0) circle (1);
	
	\foreach \a in {-2, -1, ..., 8}
	\draw [yshift=-3cm] (\a*1.732,0) circle (1);
	
	\foreach \a in {-2, -1, ..., 8}
	\draw [xshift=0.866cm,yshift=4.5cm] (\a*1.732,0) circle (1);
	
	\foreach \a in {-2, -1, ..., 8}
	\draw [xshift=0.866cm,yshift=-4.5cm] (\a*1.732,0) circle (1);
	
	\foreach \a in {0, 60, ..., 350}
	\draw [xshift=5.196cm,red,densely dashed, thick]  (\a:1.732) -- (\a+60:1.732);
	\end{scope}
	\end{tikzpicture}
	
	\caption{An optimum covering of the plane by circular disks where centers of the disks are arranged in a regular hexagonal lattice}
	\label{fig-kershner}
\end{figure}
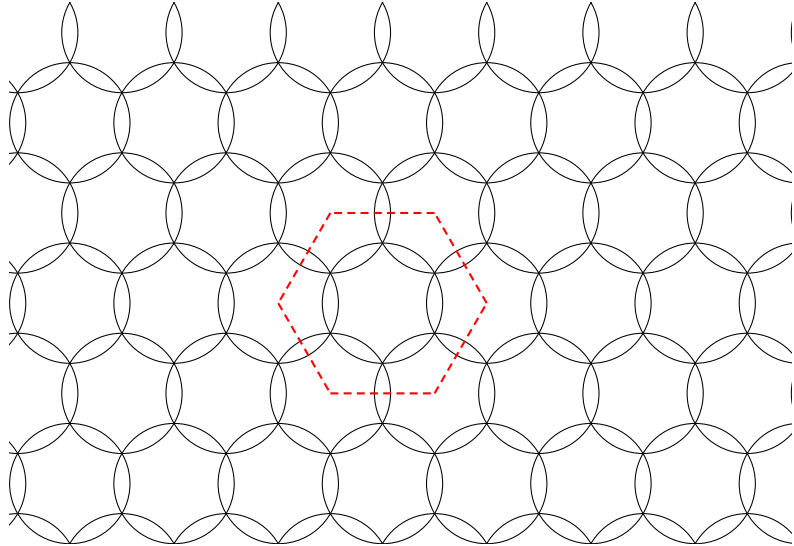
Later L. Fejes T\'{o}th provided a far-reaching generalization of this result and its counterpart for optimal circle packing in the plane, originally proved by Thue in 1892 \cite{thue1892om} and again in 1910 \cite{thue1910dichteste} (as cited in \cite{PA}), for arbitrary convex, compact sets in $\mathbb{R}^2$ with non-empty interiors.

We say two closed bounded convex sets $S$ and $T$ in $\mathbb{R}^2$ \emph{intersect simply} (refer to \cite{BR}) or \emph{do not cross each other} if they satisfy one of the conditions:
	\begin{enumerate}[(i)]
	\item $\INT(S) \cap \INT(T) = \emptyset$;
	\item $S \subset T$ or vice versa;
	\item the boundary of $S \cap T$ can be split up into two non-overlapping connected curves, one belonging to the boundary of $S$ and the other to that of $T$.
\end{enumerate}
We observe that $S$ and $T$ intersect simply if and only if both $S - T$ and $T - S$ are connected. Also two congruent copies of a circular disk always intersect each other simply.

\begin{THEOREM}\cite{MR38086}
	If $\K$ is a convex hexagon completely covered by $n$ congruent copies of a convex, compact set $\C$ (with a non-empty interior) 
	such that every pair of sets intersect each other simply,
	then 
	$n \ge \frac{a(\K)}{a(\Ply_6)},$
	where $\Ply_6$ denotes a hexagon of maximum area inscribed in $\C$.
\end{THEOREM}
This theorem leads to the following corollary, which, in turn, proves the bound on $\gamma (\Dc)$ as found by Kershner.
\begin{COROLLARY}\label{FT-bound}
	The optimal lower density of covering the plane by 
	congruent copies of the convex, compact set $\C$ with a non-empty interior, such that every pair of sets intersect each other simply, is at least the ratio of the area of $\C$ and the area of a hexagon of maximum area inscribed in $\C$, i.e. 
	$\gamma (\C) \ge \frac{a(\C)}{a(\Ply_6)}.$
\end{COROLLARY}

We are interested in the following problem of finding an optimal covering of the plane by a sequence of unit circular disks with a given constraint on the distance between the centers of consecutive disks
\footnote{Problem~\ref{prob-1} is motivated by a practical problem of robotic exploration with discrete and limited visibility discussed in \cite{GB,GBBS}. The problem of robotic exploration asks for a cost-effective strategy for a point robot to completely explore an unknown planar region avoiding obstacles. Due to mechanical limitations the robot can only scan a ball of finite and fixed radius around it at a given instance. Also it is not feasible to perform the scanning process continuously, i.e. the robot must take a break between two scans. Along with that after performing the first scan, the robot must stay inside the region that is scanned by that time (i.e. the region known to the robot at that instance).}.

\begin{PROBLEM}\normalfont\label{prob-1}
	What is the optimal lower density of covering the plane by a sequence of unit circular disks with the constraint that each disk contains the center of the next disk? Does there exist a sequence covering with the given constraint attaining the optimal lower density?
\end{PROBLEM}

We answer this question and prove Theorem~\ref{prop-1} using ideas by Fejes T\'{o}th in the case where the following restriction holds: if $(\Dc_n)_{n \in \mathbb N}$ is a sequence covering of the plane by unit circular disks with constraint mentioned in Problem~\ref{prob-1} and for all $n \ge 1$, $C_n$ is the center of the disk $\Dc_n$, then for all $n \ge 2$, measure of (smaller of the angles) $\angle C_{n-1}C_nC_{n+1}$ is at least $\frac{2\pi}{3}$. By $\angle C_{n-1}C_nC_{n+1}$ we always mean the smaller (i.e. the one with measure less than $\pi$) of the two angles (or either, if measure of $\angle C_{n-1}C_nC_{n+1}$ is $\pi$). Also abusing the notation, we use $\angle C_{n-1}C_nC_{n+1}$ to mean both the angle and also measure of the angle depending on the context.

\begin{THEOREM}\label{prop-1}
		If $\gamma^*(\Dc)$ is the optimal lower density of covering the plane by a sequence $(\Dc_n)_{n \in \mathbb N}$ of  
		unit circular disks (with $C_n$ being the center of the disk $\Dc_n$ for all $n \in \mathbb N$) such that for $n \ge 2$, $C_n$ lies in $\Dc_{n-1}$ and $\angle C_{n-1}C_nC_{n+1} \ge \frac{2\pi}{3}$, then 
	$$\gamma^*(\Dc) \ge \frac{2\pi}{2 + \sqrt{3}}.$$
	This bound can also be achieved.
\end{THEOREM}
As discussed in detail in Section~4, it seems that the restriction on $\angle C_{n-1}C_nC_{n+1}$ is only due to the limitations of the proof-techniques used in this article and the bound on the optimal lower density mentioned in Theorem~\ref{prop-1} would hold true even after removing the restriction on $\angle C_{n-1}C_nC_{n+1}$. We propose the following conjecture.
\begin{CONJECTURE}\label{conj}
	If $\gamma'(\Dc)$ is the optimal (lower) density of 
	covering the plane by a sequence of unit circular disks such that each disk in the sequence contains the center of the next disk,
	then $$\gamma'(\Dc) \ge \frac{2\pi}{2 + \sqrt{3}} \approx 1.68357.$$
\end{CONJECTURE}

In Section~3 we consider the following problem closely related to Problem~\ref{prob-1}. 
\begin{PROBLEM}\normalfont\label{prob-2}
	What is the optimal lower density of covering the plane by closed unit circular disks with the constraint that each disk contains the centers of at least two other disks? Does there exist a covering with the aforementioned constraint attaining the optimal lower density?
\end{PROBLEM}
We prove the following result in case of \emph{lattice coverings} (i.e. when centers of the disks form a lattice in ${\mathbb R}^2$).
\begin{THEOREM}\label{thm1.7}
	The optimal lower density of covering of the plane by closed unit circular disks with the constraint that the centers of the disks form a lattice in ${\mathbb R}^2$ and each disk contains the centers of at least two other disks is $\frac{2\pi}{2 + \sqrt{3}}$. The optimal lower density can be achieved by a suitable lattice covering with given constraint.
\end{THEOREM}

We note that the constraint mentioned in Problem~\ref{prob-2} is comparable to the constraint in Problem~\ref{prob-1}. In Problem~\ref{prob-1}, (second disk onward) each disk in the sequence covering contains the centers of both the previous and the next disk. We also observe that the optimal density for lattice covering of the plane with the constraint mentioned in Problem~\ref{prob-2} (as per Theorem~\ref{thm1.7}) coincides with the optimal lower density mentioned in Conjecture~\ref{conj}. This observation strengthens the case for Conjecture~\ref{conj}.

\section{Covering the plane by sequence of disks each having center in the previous one}
In this section we prove Theorem~\ref{prop-1}. First we prove that for a given covering of the plane by unit circular disks, we can partition the plane into convex bounded polygonal regions in such a way that each disk contains exactly one of the polygons.

\begin{LEMMA}\label{lemma-poly}
	Let $\Pi$ be a convex polygon in $\mathbb {R}^2$ and $(\Dc_1, \ldots, \Dc_N)$ be a finite sequence covering of $\Pi$ by unit circular disks. If $\Dc_n \ne \Dc_m$ for $n \ne m$, 
	then there is a finite sequence of convex polygons $(\Ply_1,\ldots, \Ply_N)$ such that
	\begin{enumerate}[(i)]
		\item for all $n \in \{1,\ldots,N\}$, $\Ply_n \subset \Dc_n$;
		\item $\Pi = \bigcup_{n=1}^N \Ply_n$;
		\item for all $n,m \in \{1,\ldots, N\}$ and $n \ne m$, $\INT (\Ply_n) \cap \INT (\Ply_m) = \emptyset$.
	\end{enumerate}
	\textnormal {(Note that $\Ply_n$ is allowed to be the empty set for some $n \in \{1,\ldots N\}$.)}
\end{LEMMA}
\begin{proof}
	We construct a Voronoi diagram \cite{weisensteinVoronoiDiagram} by setting the boundary of $\Pi$ as the boundary of the diagram and the centers of the circular disks as the prescribed points (known as \emph{seeds} or \emph{generators}). Suppose that for all $n \in \{1,\ldots,N\}$, $C_n$ is the center of the disk $\Dc_n$ and  
	for all $x \in \Pi$, $d_n(x)$ denotes the Euclidean distance between $x$ and $C_n$. For all $n \in \{1,\ldots,N\}$, we define a set $\Ply_n \subset \Pi$ (called the Voronoi cell associated to the seed $C_n$) as follows
	$$\Ply_n = \left\{x \in \Pi : d_n(x) = \min \{d_1(x), \ldots , d_N(x)\}\right\}.$$
	We refer to Figure~\ref{fig-voro} for an illustration.
	 
	 	\begin{figure}[!ht]
	 	\centering
	 	\[\begin{tikzpicture}
	 	\begin{scope}[scale=0.8, transform shape]		
		
	 	\draw[fill=yellow, fill opacity=0.2](0.82,0.96)--(-0.7,.95)--(-1.33,-0.21)--(-.25,-1.35)--(.25,-1.35)--(1.23,-.6)--(0.82,0.96);
	 	
	 	\draw[fill=brown, fill opacity=0.2](0.82,0.96)--(-0.7,.95)--(-1.42,2.25)--(-.8,2.4)--(.69,2.18)--(1.25,1.61)--(0.82,0.96);
	 	
	 	\draw[fill=gray, fill opacity=0.2](1.23,-.6)--(0.82,0.96)--(1.25,1.61)--(2.48,.39)--(2.28,-.84)--(1.23,-.6);
	 	
	 	\draw[fill=teal, fill opacity=0.2](.25,-1.35)--(1.23,-.6)--(2.28,-.84)--(2,-2.48)--(1.06,-2.88)--(.25,-1.35);
	 	
	 	\draw[fill=pink, fill opacity=0.2](1.06,-2.88)--(.25,-1.35)--(-.25,-1.35)--(-1.33,-2.81)--(0.15,-3.26)--(1.06,-2.88);
	 	
	 	\draw[fill=green, fill opacity=0.2](-.25,-1.35)--(-1.33,-2.81)--(-2.57,-2.44)--(-3,-1.72)--(-1.76,-0.25)--(-1.33,-0.21)--(-.25,-1.35);
	 	
	 	\draw[fill=orange, fill opacity=0.2](-3,-1.72)--(-1.76,-0.25)--(-3.23,0.94)--(-3.91,-.31)--(-3,-1.72);
	 	
	 	\draw[fill=cyan, fill opacity=0.2](-1.76,-0.25)--(-3.23,0.94)--(-2.67,1.94)--(-1.42,2.25)--(-0.7,0.95)--(-1.33,-.21)--(-1.76,-0.25);

\draw[blue, thick](2.48,0.39)--(0.69,2.18)--(-0.8,2.4)--(-2.67,1.94)--(-3.91,-0.31)--(-2.57,-2.44)--(0.15,-3.26)--(2,-2.48)--(2.48,0.39);

\node[circle, draw, red, fill=red, inner sep=0pt, minimum size=1pt] at (0,0) {};
\draw[red, densely dashed] (0,0) circle (1.38);

\node[circle, draw, red, fill=red, inner sep=0pt, minimum size=1pt] at (2.09,0.49) {};
\draw[red, densely dashed] (2.09,0.49) circle (1.38);

\node[circle, draw, red, fill=red, inner sep=0pt, minimum size=1pt] at (-0.06,1.96) {};
\draw[red, densely dashed] (-0.06,1.96) circle (1.38);

\node[circle, draw, red, fill=red, inner sep=0pt, minimum size=1pt] at (-1.88,0.95) {};
\draw[red, densely dashed] (-1.88,0.95) circle (1.38);

\node[circle, draw, red, fill=red, inner sep=0pt, minimum size=1pt] at (-2.95,-0.36) {};
\draw[red, densely dashed] (-2.95,-0.36) circle (1.38);

\node[circle, draw, red, fill=red, inner sep=0pt, minimum size=1pt] at (-1.64,-1.46) {};
\draw[red, densely dashed] (-1.64,-1.46) circle (1.38);

\node[circle, draw, red, fill=red, inner sep=0pt, minimum size=1pt] at (0.06,-2.7) {};
\draw[red, densely dashed] (0.06,-2.7) circle (1.38);

\node[circle, draw, red, fill=red, inner sep=0pt, minimum size=1pt] at (1.49,-1.96) {};
\draw[red, densely dashed] (1.49,-1.96) circle (1.38);

	 	\end{scope}
	 	\end{tikzpicture}\]
	 	\caption{Eight circular disks covering a convex octagon and the Voronoi cells associated to the centers of the disks}
	 	\label{fig-voro}
	 \end{figure}
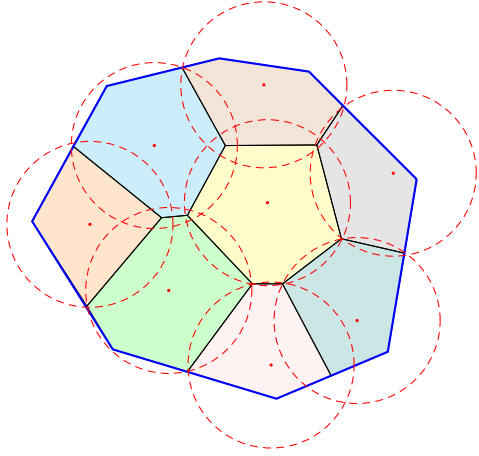
	 
	 We note that $\Ply_n$ may be the empty set for some $n$ (only if $C_n$ lies outside $\Pi$) and we also replace $\Ply_n$ by the empty set if $\Ply_n$ contains no interior points of $\Pi$.
	 
	 From the properties of Voronoi cells (in a finite dimensional Euclidean space), it follows that each $\Ply_n$ is a convex polygon (or the empty set). Also $\cup_{n=1}^N \Ply_n = \Pi$ and for $n,m \in \{1,\ldots,N\}$ with $n \ne m$, $\INT (\Ply_n) \cap \INT (\Ply_m) = \emptyset$. 
	 
	 Let $\Dc_n$ and $\Dc_m$ be two distinct disks whose interiors intersect. 
	It follows that $\Ply_n$ is contained in one of the two closed half-planes induced by 
	the line corresponding to the common chord of $\Dc_n$ and $\Dc_m$ and $\Ply_m$ is contained in the other closed half-plane. This implies, since $\{\Dc_1, \ldots, \Dc_N\}$ covers $\Pi$, $\Ply_n \subset \Dc_n$ for all $n \in \{1,\ldots,N\}$.
\end{proof}

We now consider a convex subset of a closed unit circular disk, whose boundary consists of 
	\begin{enumerate}[(i)]
	\item any pair of chords of the unit circle (boundary of the circular disk), of length $\sqrt 3$ each, sharing a common endpoint and
	\item the minor arc of the unit circle between the other two endpoints (i.e. other than the common endpoint) of the aforementioned two chords.
	\end{enumerate}
We denote the compact, convex subset of the closed unit disk as described above by $\M$ (\emph{up to congruence}). For an illustration, we refer to Figure~\ref{fig-M}. Let $A_0B_0$ and $A_0B'_0$ be any two chords of the unit circle (on the left) of length $\sqrt{3}$ each. $\M$ is the convex, compact set whose boundary consists of the chords $A_0B_0$ and $A_0B'_0$ and the (minor) arc $B_0B'_0$ of the circle.

Next we consider a convex subset of a closed unit circular disk, whose boundary consists of 
\begin{enumerate}[(i)]
	\item any two chords of the unit circle, of length $\sqrt 3$ each, that do not cross each other (they may at most meet at a point on the circle), where the angle subtended by the line segment joining the centers of the two chords from the center of the disk is $\theta$;
	\item the (minor) arc between an ``alternative-pair" of endpoints (i.e. one endpoint from each chord) such that the distance between the pair is the least among all possible alternative-pairs (note that the arc may be of length 0, possible if and only if the chords meet at a point);
	\item the (minor) arc between the other two endpoints (i.e. other than the closest alternative-pairs as mentioned in (ii)) of the aforementioned two chords.
\end{enumerate}
We denote the compact, convex subset of the unit disk as described above by $\M'_\theta$ (up to congruence). For an illustration, we again refer to Figure~\ref{fig-M}. Let $AB$ and $A'B'$ be any two non-crossing chords of the unit circle (on the right) of length $\sqrt{3}$ each and $C$ be the center of the circle. If $D$ and $D'$ are the midpoints of the chords $AB$ and $A'B'$ respectively, then we have $\theta = \angle DCD'$. $\M'_\theta$ is the convex, compact set whose boundary consists of the chords $AB$ and $A'B'$ and the (minor) arcs $AA'$ and $BB'$ of the circle. 
	
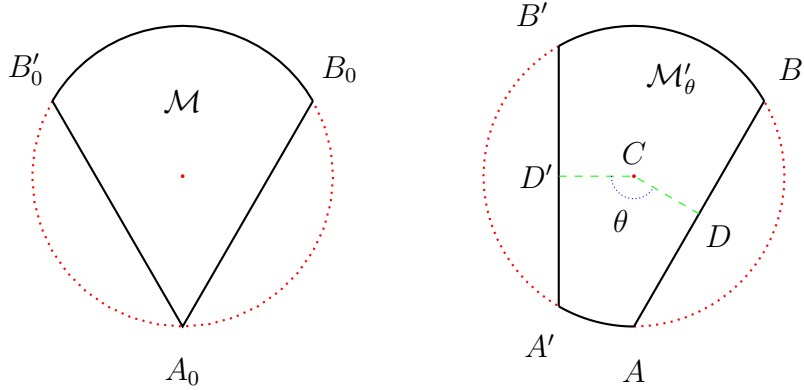
\begin{figure}[!ht]
	
	\centering
	\[\begin{tikzpicture}
	\node[circle, draw, red, fill=red, inner sep=0pt, minimum size=1pt] at (-6,0) {};	
	
	\draw [thick,domain=30:150] plot ({-6+2*cos(\x)}, {2*sin(\x)});
	
	\draw [dotted,red,thick,domain=-210:30] plot ({-6+2*cos(\x)}, {2*sin(\x)});
	
	\draw [thick](-4.268,1)--(-6,-2)--(-7.732,1);
	
	\node at ($(-6,0)+(90:1 and 1)$) {$\M$};
	
	\node at ($(-6,-2)+(-90:0.5 and .6)$) {$A_0$};
	\node at ($(-4.268,1)+(45:0.5 and .6)$) {$B_0$};
	\node at ($(-7.732,1)+(135:0.5 and .6)$) {$B'_0$};
	
	\node[circle, draw, red, fill=red, inner sep=0pt, minimum size=1pt] at (0,0) {};
	
	\draw [thick,domain=30:120] plot ({2*cos(\x)}, {2*sin(\x)});
	\draw [thick,domain=-120:-90] plot ({2*cos(\x)}, {2*sin(\x)});
	
	\draw [dotted,red,thick,domain=-90:30] plot ({2*cos(\x)}, {2*sin(\x)});
	\draw [dotted,red,thick,domain=120:240] plot ({2*cos(\x)}, {2*sin(\x)});
	
	\draw [thick](1.732,1)--(0,-2);
	\draw [thick](-1,1.732)--(-1,-1.732);
	
	\draw [dashed,green](-1,0)--(0,0)--(0.866,-0.5);
	
	\draw [densely dotted,blue,domain=180:330] plot ({0.3*cos(\x)}, {0.3*sin(\x)});
	\node at ($(0,0)+(-100:1 and .6)$) {$\theta$};
	
	\node at ($(0,0)+(60:1 and 1.5)$) {$\M'_\theta$};
	
	\node at ($(0,-2)+(-90:0.5 and .6)$) {$A$};
	\node at ($(1.732,1)+(45:0.5 and .6)$) {$B$};
	\node at ($(-1,-1.732)+(-120:0.5 and .6)$) {$A'$};
	\node at ($(-1,1.732)+(135:0.5 and .6)$) {$B'$};
	
	\node at ($(0,0)+(90:1 and .3)$) {$C$};
	\node at ($(0.866,-0.5)+(-30:0.3 and .6)$) {$D$};
	\node at ($(-1,0)+(180:0.3 and .3)$) {$D'$};
	\end{tikzpicture}\]
	\caption{Convex, compact sets $\M$ and $\M'_\theta$}
	\label{fig-M}
\end{figure}

Since the pair of chords which are parts of the boundary of $\M'_\theta$ don't cross each other, we have $\frac{2\pi}{3} \le \theta \le \pi$. Also we have $\M'_{\frac{2\pi}{3}} = \M$ and for $\frac{2\pi}{3} \le \theta \le \pi$, $a(\M'_\theta) = a(\M)$.

\begin{OBSERVATION}\label{obs-1}\normalfont
	Let us consider the convex, compact set $\M'_\theta$ for an arbitrary $\theta \in \left[\frac{2\pi}{3}, \pi\right]$. Let $\M'_\theta$ be bounded by the chords $AB$ and $A'B'$ and the (minor) arcs $AA'$ and $BB'$ as illustrated in Figure~\ref{fig-M} and $\Ply$ be a proper $\nu$-gon of the largest area inscribed in $\M'_\theta$. We have the following.
	\begin{enumerate}[(i)]
	\item $\Ply$ has vertices on both the line segments $AB$ and $A'B'$.
	\item If $\Ply$ has two vertices on the line segment $AB$ (or $A'B'$), then the points $A$ and $B$ ($A'$ and $B'$) are the two vertices of $\Ply$ on $AB$.
	\item If $\Ply$ has exactly one vertex $Q$ on the line segment $AB$ (or $A'B'$), 
	then there exists a proper $\nu$-gon $\Ply'$ with the same area inscribed in $\M'_\theta$ such that $\Ply'$ has exactly one of $A$ or $B$ as its vertex.
	Let $P$ and $R$ be the neighbors of $Q$ (considering the $\nu$-gon $\Ply$ as a cycle graph), then $\triangle PQR$, taking $PR$ to be the base, attains the maximum height (and hence maximum area) when $Q$ is one of the endpoints on $AB$, unless $PR$ is parallel to $AB$. 
	In that case, we replace the vertex $Q$ of $\Ply$ by any of the points $A$ or $B$ to obtain $\Ply'$.
	\item If $\Ply$ has at least two vertices on the arc $AA'$ (or $BB'$), then $A$ and $A'$ ($B$ and $B'$ ) are two vertices of $\Ply$. Otherwise let $A$ not be a vertex of $\Ply$ and $Q$ and $R$ (with $Q \ne A$) be the ``first" two vertices of $\Ply$ on the arc  $AA'$ in a cyclic order (from $A$ toward $A'$) of the vertices $\Ply$. Also let $P$ (with $P \ne A$) be the vertex ``preceding" $Q$ in the same cyclic order. 
	From (i) -- (iii) it follows that $P$ is the only vertex of $\Ply$ on $AB$ and we may replace the vertex $P$ of $\Ply$ by the point $B$ to obtain a proper $\nu$-gon $\Ply'$ with the same area inscribed in $\M'_\theta$. We note that $\triangle BQR$, taking $BR$ to be the base, attains its maximum height (and hence maximum area) when $Q$ coincides with $A$. Thus by replacing the vertex $Q$ of $\Ply'$ by the point $A$, we obtain another proper $\nu$-gon $\Ply''$ inscribed in $\M'_\theta$ with $a(\Ply'') > a(\Ply')=a(\Ply)$. This leads to a contradiction to the assumption that $\Ply$ is a largest $\nu$-gon inscribed in $\M'_\theta$.
\end{enumerate}
\end{OBSERVATION}

\begin{OBSERVATION}\label{rem-0}\normalfont
	From Observation~\ref{obs-1} it follows that the largest $\nu$-gon inscribed in $\M$, where $\M$ is bounded by the chords $A_0B_0$ and $A'_0B'_0$ and the (minor) arc $B_0B'_0$ as illustrated in Figure~\ref{fig-M}, is the convex $\nu$-gon with vertices $A_0,B_0,B'_0$ and $\nu - 3$ points on the arc $B_0B'_0$ (other than $B_0$ and $B'_0$) that divide the arc into $\nu - 2$ pieces of equal length.
\end{OBSERVATION}

Over the next two lemmata we show that a polygonal cell constructed in the proof of Lemma~\ref{lemma-poly}, when the disks abide by the constraints mentioned in Theorem~\ref{prop-1}, is always smaller than the largest polygon with same number of vertices inscribed in $\M$. First we show that even when a largest $\nu$-gon inscribed in $\M'_\theta$, for an arbitrary $\theta \in \left[\frac{2\pi}{3}, \pi\right]$, doesn't fit inside $\M$, we may break the $\nu$-gon into triangular pieces and rearrange them appropriately to obtain a $\nu$-gon (of the same area) that fits inside $\M$.
\begin{LEMMA}\label{M-bigger-poly}
	If $\Ply$ is a proper $\nu$-gon of the largest area inscribed in $\M'_\theta$ for an arbitrary $\theta \in \left[\frac{2\pi}{3}, \pi\right]$, then there is a proper $\nu$-gon $\Ply^*$ contained in $\M$ such that $a(\Ply) \le a(\Ply^*)$.	
\end{LEMMA}
\begin{proof}
	Let the boundary of $\M'_\theta$ consist of the line segments $AB$ and $A'B'$ and the arcs of the unit circle (with the center $C$) $AA'$ and $BB'$ as illustrated in Figure~\ref{fig-M}. 
	\indent \emph {Case 1:} $\Ply$ has both $AB$ and $A'B'$ as two of its sides.\\
	In this case we can construct a $\nu$-gon $\Ply^*$ which can be inscribed in $\M$ having the same area as of $\Ply$ by removing all but one vertices of $\Ply$ from one of the arcs and repositioning them along the boundary of the unit circle in an appropriate manner. 
	
	Let $A=P_0, P_1, \ldots, P_k=A'$ be the proper vertices of $\Ply$ on the arc $AA'$ in cyclic order (from $A$ toward $A'$). We remove the points $P_1, \ldots, P_k=A'$ from the arc $AA'$; place them along the (minor) arc $B'A'$ of the unit circle in cyclic order (from $B'$ toward $A'$) and label them as $P_1^*, \ldots, P_k^*$, respectively such that $\angle B'CP_1^* = \angle P_0CP_1 = \angle ACP_1$ and $\angle P_iCP_{i+1} = \angle P_i^*CP_{i+1}^*$ for all $i \in \{1, \ldots, k-1\}$ (as illustrated for $k=2$ in Figure~\ref{fig-poly-switch}).
	
	\begin{figure}[!ht]
	\centering
	\[\begin{tikzpicture}
	
	\draw[loosely dotted,blue,thick,domain=-210:30] plot ({6+2*cos(\x)}, {2*sin(\x)});
	
	\draw [densely dotted,red,thick,domain=30:150] plot ({6+2*cos(\x)}, {2*sin(\x)});
	
	\draw [densely dotted,red,thick](7.732,1)--(6,-2)--(4.268,1);
	\draw [dashed,xshift=6cm](30:2)--(60:2)--(90:2)--(120:2)--(135:2)--(150:2)--(270:2)--(30:2);
	
	\draw[dashed,green,thick,xshift=6cm](0:0)--(30:2);
	\draw[dashed,green,thick,xshift=6cm](0:0)--(60:2);
	\draw[dashed,green,thick,xshift=6cm](0:0)--(90:2);
	\draw[dashed,green,thick,xshift=6cm](0:0)--(120:2);
	\draw[dashed,green,thick,xshift=6cm](0:0)--(135:2);
	\draw[dashed,green,thick,xshift=6cm](0:0)--(150:2);
	\draw[dashed,green,thick,xshift=6cm](0:0)--(270:2);
	
	\node at ($(6,0)+(90:1 and 1)$) {$\Ply^*$};
	
	\node at ($(6,-2)+(-90:0.5 and .6)$) {$A$};
	\node at ($(7.732,1)+(45:0.5 and .6)$) {$B$};
	\node at ($(6,0)+(0:0.5 and .3)$) {$C$};
	
	\node at ($(5,1.732)+(120:0.5 and .3)$) {$B'$};
	\node at ($(4.586,1.414)+(135:0.5 and .3)$) {$P_1^*$};
	\node at ($(4.268,1)+(150:0.5 and .3)$) {$P_2^*$};

	\draw[loosely dotted,blue,thick,domain=120:240] (0,0) plot ({2*cos(\x)}, {2*sin(\x)});
	\draw[loosely dotted,blue,thick,domain=-90:30] (0,0) plot ({2*cos(\x)}, {2*sin(\x)});
	
	\draw [densely dotted,red,thick,domain=30:120] plot ({2*cos(\x)}, {2*sin(\x)});
	\draw [densely dotted,red,thick,domain=-120:-90] plot ({2*cos(\x)}, {2*sin(\x)});
	
	\draw [densely dotted,red,thick](1.732,1)--(0,-2);
	\draw [densely dotted,red,thick](-1,1.732)--(-1,-1.732);
	
	\draw [dashed](30:2)--(60:2)--(90:2)--(120:2)--(240:2)--(255:2)--(270:2)--(30:2);
	
	\draw[dashed,green,thick](0:0)--(30:2);
	\draw[dashed,green,thick](0:0)--(60:2);
	\draw[dashed,green,thick](0:0)--(90:2);
	\draw[dashed,green,thick](0:0)--(120:2);
	\draw[dashed,green,thick](0:0)--(240:2);
	\draw[dashed,green,thick](0:0)--(255:2);
	\draw[dashed,green,thick](0:0)--(270:2);

	\node at ($(0,0)+(60:1 and 1.5)$) {$\Ply$};
	
	\node at ($(0,0)+(0:0.5 and .3)$) {$C$};
	
	\node at ($(0,-2)+(-30:0.8 and .6)$) {$P_0=A$};
	\node at ($(1.732,1)+(45:0.5 and .6)$) {$B$};
	\node at ($(-1,-1.732)+(-160:0.5 and 1)$) {$A'=P_2$};
	\node at ($(-1,1.732)+(135:0.5 and .6)$) {$B'$};
	\node at ($(255:2)+(275:0.8 and .3)$) {$P_1$};
	
	\end{tikzpicture}\]
	\caption{Construction of $\Ply^*$ from $\Ply$}
	\label{fig-poly-switch}
\end{figure}
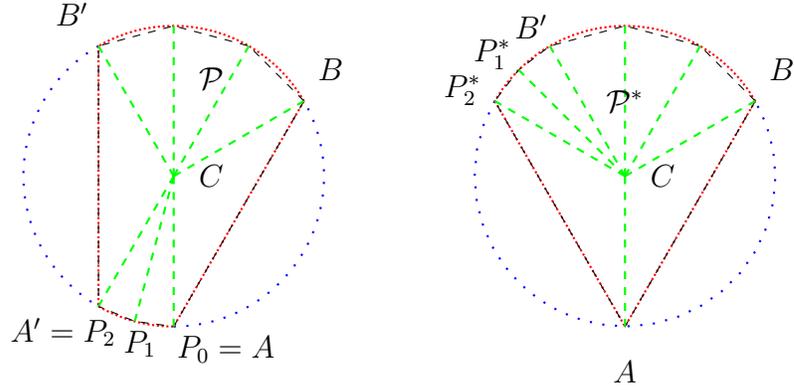
Hence, if $\Ply^*$ is the convex proper $\nu$-gon whose set of proper vertices is 
$$\{\mbox{all proper vertices of }\Ply \} \cup \{P_1^*, \ldots, P_k^*\} - \{P_1, \ldots, P_k=A'\},$$ 

then $a(\Ply^*) = a(\Ply)$. Also we note that $\Ply^*$ can be inscribed in $\M$ as the compact, convex set whose boundary consists of the line segments $AB$ and $AP_k^*$, and the arc $BB'P_k^*$ is $\M$ (up to congruence).

\indent \emph {Case 2:} At least one of $AB$ and $A'B'$ is not a side of $\Ply$.\\
We note that, by Observation~\ref{obs-1}, in this case $\Ply$ has only one vertex, say $P$, on one of the arcs $AA'$ and $BB'$, without loss of generality, say on $AA'$. It also follows from Observation~\ref{obs-1} that the neighbors of $P$ (considering the proper $\nu$-gon $\Ply$ as a cycle graph) are $B$ and $B'$. We observe that $\Ply$ itself can be inscribed in $\M$. $\Ply$ is a subset of 
the compact, convex set whose boundary consists of the pair of chords of length $\sqrt{3}$ with a common endpoint in $P$ (as illustrated by dashed line segments in Figure~\ref{fig-poly-thin}) and the (minor) arc joining the other two endpoints of these two chords and this set, up to congruence, is $\M$.

	\begin{figure}[!ht]
	\centering
	\[\begin{tikzpicture}
	\begin{scope}[scale=1.2, transform shape]		
	
	\draw[loosely dotted,blue,thick,domain=140:260] (0,0) plot ({2*cos(\x)}, {2*sin(\x)});
	\draw[loosely dotted,blue,thick,domain=-80:40] (0,0) plot ({2*cos(\x)}, {2*sin(\x)});
	\draw [densely dotted,red,thick,domain=40:140] plot ({2*cos(\x)}, {2*sin(\x)});
	\draw [densely dotted,red,thick,domain=-100:-80] plot ({2*cos(\x)}, {2*sin(\x)});
	
	\draw[densely dotted,red,thick](140:2)--(260:2);
	\draw[densely dotted,red,thick](-80:2)--(40:2);
	
	\draw [](140:2)--(270:2)--(40:2)--(65:2)--(90:2)--(115:2)--(140:2);
	\draw [dashed,green,thick](150:2)--(270:2)--(30:2);
	
	\node at ($(0,0)+(60:1 and 1.5)$) {$\Ply$};
	
	\node at ($(-90:2)+(-90:0.5 and .3)$) {$P$};
	\node at ($(-80:2)+(-80:0.5 and .3)$) {$A$};
	\node at ($(1.732,1)+(80:0.5 and .6)$) {$B$};
	\node at ($(-100:2)+(-100:0.5 and .3)$) {$A'$};
	\node at ($(-1.732,1)+(100:0.5 and .6)$) {$B'$};
	
	\end{scope}
	\end{tikzpicture}\]
	\caption{$\Ply$ is bounded between chords of length $\sqrt{3}$ that meet at $P$ (dashed lines)}
	\label{fig-poly-thin}
\end{figure}
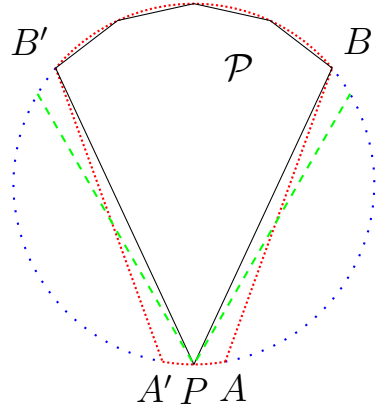
\end{proof}

Let $\Pi$ be any convex polygon in $\mathbb {R}^2$ and $(\Dc_1, \ldots, \Dc_N)$ be a finite sequence covering of $\Pi$ by closed unit circular disks such that for all $n \in \{1, \ldots, N\}$, $C_n$ is the center of the disk $\Dc_n$. Also suppose that for all $n \in \{2, \ldots, N\}$, $C_n \in \Dc_{n-1}$ and for all $n \in \{2, \ldots, N-1\}$, $\angle C_{n-1}C_nC_{n+1} \ge \frac{2\pi}{3}$. 

Let $\Ply_n$, for $n \in \{1, \ldots, N\}$, be the Voronoi cell (a convex polygon) associated to $C_n$ obtained from the Voronoi diagram with the boundary of $\Pi$ as the boundary of the diagram and $C_1, \ldots, C_N$ as the generators (refer to Lemma~\ref{lemma-poly}). Suppose that for all $n \in \{1, \ldots, N\}$, the $\Ply_n$ has $\nu_n$ proper vertices.

From the notion of Voronoi cells, we note that the cell $\Ply_n$ ($\subset \Dc_n$) is bounded between the common chord of $\Dc_{n-1}$ and $\Dc_{n}$, and the common chord of $\Dc_{n}$ and $\Dc_{n+1}$. This, in turn, implies $\Ply_n$ can be fit inside $\M'_\theta$ for some $\theta \in \left[\frac{2\pi}{3}, \pi\right]$ (in fact, $\theta = \angle C_{n-1}C_nC_{n+1}$ in this case).
Then by using Lemma~\ref{M-bigger-poly}, we have the following bound on $a(\Ply_n)$.
\begin{LEMMA}\label{lemma-bigger-poly}
	If, for all $n \in \{2, \ldots, N-1\}$, $\Ply^*_{\nu_n}$ is the $\nu_n$-gon with the largest area inscribed in $\M$, then $a(\Ply_n) \le a(\Ply^*_{\nu_n})$.
		
\end{LEMMA}

\begin{proof}
	From the construction of the convex $\nu_n$-gon $\Ply_n$ that lies inside $\Dc_n$ (for $n \in \{2, \ldots, N-1\}$) in the proof of Lemma~\ref{lemma-poly}, it follows that  $\Ply_n$ is bounded between the common chord of $\Dc_{n-1}$ and $\Dc_{n}$, and the common chord of $\Dc_{n}$ and $\Dc_{n+1}$ (as illustrated by solid line segments in Figure~\ref{fig-bdd}). Since, $C_n$, the center of $\Dc_{n}$, lies in $\Dc_{n-1}$, the length of the common chord of $\Dc_{n-1}$ and $\Dc_{n}$ is at least $\sqrt{3}$ (and the same is true for the common chord of $\Dc_{n}$ and $\Dc_{n+1}$). 
	Now we consider the following pair of chords of $\Dc_n$ of length $\sqrt{3}$ each: 
	\begin{enumerate}[(i)]
		\item one that is parallel to the common chord of $\Dc_{n-1}$ and $\Dc_{n}$ and that lies in the minor segment of $\Dc_n$ induced by the common chord of $\Dc_{n-1}$ and $\Dc_{n}$;
		\item one that is parallel to the common chord of $\Dc_{n}$ and $\Dc_{n+1}$ and that lies in the minor segment of $\Dc_n$ induced by the common chord of $\Dc_{n}$ and $\Dc_{n+1}$.
	\end{enumerate}
	(these two chords are illustrated by the dash-dotted line segments in Figure~\ref{fig-bdd}).

	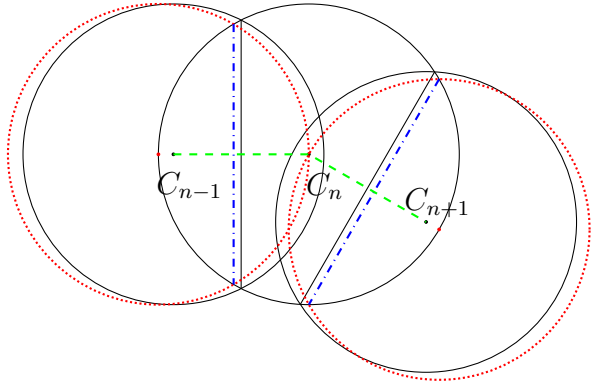
\begin{figure}[!ht]
	\centering
	\[\begin{tikzpicture}
	\node[circle, draw, fill=black, inner sep=0pt, minimum size=1pt] at (-1.8,0) {};
	\draw[] (-1.8,0) circle (2);
	
	\node[circle, draw, fill=black, inner sep=0pt, minimum size=1pt] at (0,0) {};
	\draw[] (0,0) circle (2);
	
	\node[circle, draw, fill=black, inner sep=0pt, minimum size=1pt] at (1.5588,-0.9) {};
	\draw[] (1.5588,-0.9) circle (2);
	
	\draw [](-0.9,1.786)--(-0.9,-1.786);
	\draw [](1.6724,1.0968)--(-0.1136,-1.9968);

	\node[circle, draw, red, fill=red, inner sep=0pt, minimum size=1pt] at (-2,0) {};
	\draw[red,thick,densely dotted] (-2,0) circle (2);
	\node[circle, draw, red, fill=red, inner sep=0pt, minimum size=1pt] at (1.732,-1) {};
	\draw[red,thick,densely dotted] (1.732,-1) circle (2);
	
	\draw [dash dot,blue,thick](1.732,1)--(0,-2);
	\draw [dash dot,blue,thick](-1,1.732)--(-1,-1.732);
	
	\draw [dashed,green,thick](-1.8,0)--(0,0)--(1.5588,-0.9);
	
	\node at ($(-1.8,0)+(-45:.3 and .6)$) {$C_{n-1}$};
	\node at ($(0,0)+(-45:.3 and .6)$) {$C_n$};
	\node at ($(1.5588,-0.9)+(60:.3 and .3)$) {$C_{n+1}$};

	\end{tikzpicture}\]
	\caption{Common chord of $\Dc_{n-1}$ and $\Dc_{n}$, common chord of $\Dc_{n}$ and $\Dc_{n+1}$ (solid), and the chords of $\Dc_{n}$ of length $\sqrt{3}$ (dash-dotted) parallel to the common chords}
	\label{fig-bdd}
\end{figure}

It follows that $\Ply_n$ is bounded between the aforementioned pair of chords of length $\sqrt{3}$ each. We also note that as $\angle C_{n-1}C_nC_{n+1} \ge \frac{2\pi}{3}$, these two of chords of length $\sqrt{3}$ don't cross each other (they can at most meet at a point on the boundary of $\Dc_n$). Hence, the $\nu_n$-gon $\Ply_n$ can be fit inside the convex, compact set $\M'_\theta$, where $\theta = \angle C_{n-1}C_nC_{n+1}$. This is because the aforementioned pair of chords of length $\sqrt{3}$ each of $\Dc_n$ along with a pair of suitable (minor) arcs describe the boundary of $\M'_\theta$ (up to congruence).

Thus, for all $n \in \{2, \ldots, N-1\}$, the area of the $\nu_n$-gon $\Ply_n$ is less than or equal to the area of a largest $\nu_n$-gon inscribed in  $\M'_\theta$. 
Hence, from Lemma~\ref{M-bigger-poly}, it follows that $a(\Ply_n) \le a(\Ply^*_{\nu_n})$.
\end{proof}

Let $\Pi$ be any convex hexagon in $\mathbb {R}^2$ and $(\Dc_1, \ldots, \Dc_N)$ be a finite sequence covering of $\Pi$ by closed unit circular disks. Suppose that for $n \in \{1, \ldots, N\}$, $\Ply_n$ is the Voronoi cell (a convex polygon) corresponding to the disk $\Dc_n$ (in reference to Lemma~\ref{lemma-poly}). We have the following bound on the average number of (proper) vertices of these polygons.
\begin{LEMMA}\label{lemma-best6}
	If, for all $n \in \{1, \ldots, N\}$, the Voronoi cell $\Ply_n$ has $\nu_n$ proper vertices, then 
	$\sum_{n = 1}^{N} \nu_n \le 6N$.
\end{LEMMA}

\begin{proof}
	We provide a proof of the lemma following the proof of a similar result by Bambah and Rogers in
	their proof of L. Fejes T\'{o}th's Theorem in \cite{BR}.
	First we form augmented polygons $\Ply'_0, \Ply'_1, \ldots , \Ply'_N$ from $\Pi, \Ply_1, \ldots , \Ply_N$ respectively by regarding a point $x$ on the boundary of one of $\Pi, \Ply_1, \ldots , \Ply_N$ as a vertex of the corresponding polygon $\Ply'_0, \Ply'_1, \ldots , \Ply'_N$ if it is a proper vertex of at least one of the polygons $\Pi, \Ply_1, \ldots , \Ply_N$. Suppose that for all $i \in \{0, \ldots , N\}$, $\Ply'_i$ has $\nu'_i$ sides. Let $n_0, n_1$, and $n_2$ be the number of vertices, sides and regions respectively comprising the system of polygons $\Ply'_1, \ldots , \Ply'_N$. It follows that $n_2 = N$. Let $\alpha_1, \ldots, \alpha_{n_0}$ be the number of sides meeting at the different vertices of the configuration. We observe that for all $i \in \{1, \ldots, n_0\}, \alpha_i$ is at least 3, except perhaps for those vertices which are proper vertices of $\Pi$ (for such a vertex, $\alpha_i$ is at least 2). Since $\Pi$ is a hexagon, we have 
	$$\sum_{i = 1}^{n_0} \alpha_i \ge 3n_0 - 6.$$
	But counting the incident pairs of a vertex and a side in two ways we have
	$$\sum_{i = 1}^{n_0} \alpha_i  = 2n_1.$$
	By Descartes-Euler polyhedral formula \cite{weisensteinPolyhedralFormula}, viz.  $V - E + F =2$, we have $n_0 - n_1 + (n_2 + 1) = 2$ (we note that number of  faces (regions), $F$ is $n_2 + 1 = N + 1$ here, as the complement of $\Pi$ in $\mathbb R^2$ contributes 1), which implies $3n_0 - 6 = 3n_1 - 3N -3$. Thus we get $2n_1 \ge 3n_1 - 3N -3$, i.e. $n_1 \le 3N + 3$. Now counting the incident pairs of a polygon and a side in two ways we have 
	$$\sum_{i = 0}^{N} \nu'_i = 2n_1.$$
	Hence,
	$$\sum_{i = 1}^{N} \nu_i \le \sum_{i = 0}^{N} \nu'_i - 6 = 2n_1 -6 \le 6N.$$
\end{proof}

We now use the following theorem by Dowker to conclude that if, for all $n \ge 3$, $a^*(n)$ denotes the area of the largest $n$-gon inscribed in $\M$, then $a^*(n)$ is bigger than the average of $a^*(n-1)$ and $a^*(n+1)$.
\begin{THEOREM}\cite{dowker1944}\label{theorem-dowker}
	Given a convex, compact set $\C$ in $\mathbb{R}^2$ with a non-empty interior, if, for $n \ge 3$, $\Q_n$ denotes an $n$-gon of the largest area inscribed in $\C$, then 
	$$a(\Q_n) \ge \frac{a(\Q_{n-1})+a(\Q_{n+1})}{2},$$
	for every $n \ge 4$.
\end{THEOREM}

\begin{COROLLARY}\label{concave-fn}
	If, for all $n \ge 3$, $a^*(n)$ denotes the area of the largest $n$-gon inscribed in $\M$, then for every $n \ge 4$, 
	$$a^*(n) \ge \frac{a^*(n-1)+a^*(n+1)}{2}.$$
\end{COROLLARY}

\begin{proof}
	Follows from Theorem~\ref{theorem-dowker} as $\M$ is a convex, compact set in $\mathbb{R}^2$ with a non-empty interior.
\end{proof}

\begin{REMARK}\normalfont
	From Observation~\ref{rem-0} it follows that if, for all $n \ge 3$, $a^*(n)$ denotes the area of the largest $n$-gon inscribed in $\M$, then $$a^*(n) = \frac{\sqrt{3}}{2} + \frac{n-2}{2} \sin \left(\frac{2\pi}{3(n-2)}\right).$$
\end{REMARK}

We now extend the function $a^*$ to all real numbers bigger than or equal to 3 as the linear interpolation on the set of data points $\{(n, a^*(n)) : n \in \mathbb{N}, n \ge 3\}$. From Corollary~\ref{concave-fn} it follows that $a^*$ is a concave function. We have the following result for a real concave function.

\begin{THEOREM}[Jensen's inequality for concave functions, 1906]\label{theorem-jensen}
	If $f$ is a real concave function and $x_1, \ldots , x_n$ are real numbers in its domain, then 
	$$\frac{\sum_{i=1}^n f(x_i)}{n} \le f \left(\frac{ \sum_{i=1}^n x_i}{n} \right).$$
	Equality holds if and only if $x_1 = \cdots = x_n$ or $f$ is linear.
\end{THEOREM}

Now we proceed to prove the lower bound mentioned in Theorem~\ref{prop-1}.
\begin{THEOREM}
	If $\gamma^*(\Dc)$ is the optimal lower density of covering the plane by a sequence of unit circular disks $(\Dc_1, \ldots, \Dc_n, \ldots)$ (with $C_n$ being the center of the disk $\Dc_n$ for all $n \in \mathbb N$) such that for $n \ge 2$, $C_n$ lies in $\Dc_{n-1}$ and $\angle C_{n-1}C_nC_{n+1} \ge \frac{2\pi}{3}$, then 
	$$\gamma^*(\Dc) \ge \frac{2\pi}{2 + \sqrt{3}}.$$
\end{THEOREM}

\begin{proof}
	Let $\Pi$ be a convex hexagon containing the center of $\Dc_1$. Since $\Pi$ is a compact set, it follows that there is a least $N \in \mathbb{N}$ such that $\{\Dc_1, \ldots , \Dc_N\}$ covers $\Pi$. If multiple disks from the finite sequence $\Dc_1, \ldots , \Dc_N$ coincide, we throw away all but one. Let $\Ply_1, \ldots, \Ply_N$ be the polygons (Voronoi cells) corresponding to the disks $\Dc_1, \ldots , \Dc_N$, respectively as constructed in the proof of Lemma~\ref{lemma-poly} (with $\Ply_n=\emptyset$ for a ``thrown away" disk $\Dc_n$). Also suppose that for $n \in \{1,\ldots,N\}, \Ply_n$ is a proper $\nu_n$-gon. From Lemma~\ref{lemma-poly}, we have $a(\Pi) = \sum_{n = 1}^N a(\Ply_n)$.
	
	From Lemma~\ref{lemma-bigger-poly} 	we have, for $n \in \{2,\dots,N-1\}$, $a(\Ply_n) \le a^*(\nu_n)$. Also we can choose a positive number $c$ such that $a(\Ply_1) \le a^*(\nu_1) + c$ and $a(\Ply_N) \le a^*(\nu_N) + c$. Thus  
	
	\begin{align*}
	\frac{\sum_{n = 1}^N a(\Ply_n)}{N} &\le \frac{\sum_{n = 1}^N a^*(\nu_n) + 2c}{N}\\
			&\le \frac{\sum_{n = 1}^N a^*(\nu_n)}{N} + \frac{2c}{N}\\
			&\le a^*\left(\frac{\sum_{n = 1}^N \nu_n}{N} \right) + \frac{2c}{N}	 ~  \mbox{ (by Jensen's inequality (Theorem~\ref{theorem-jensen}))}.
	\end{align*}
	As $a^*$ is an increasing function, from Lemma~\ref{lemma-best6} it follows that 
	$$\frac{\sum_{n = 1}^N a(\Ply_n)}{N} \le a^*(6) + \frac{2c}{N}.$$ 
	Thus we have 
	$$\frac{\sum_{n=1}^N a(\Dc_n)}{a(\Pi)} = \frac{N\pi}{\sum_{n = 1}^N a(\Ply_n)} \ge \frac{\pi}{a^*(6) + \frac{2c}{N}}.$$
	
We note that if we replace $\Pi$ by $\lambda\Pi$ ($\lambda \ge 1$) in the inequality above, the integer $N$ depends on $\lambda$ (i.e. $N = N(\lambda)$) and $N(\lambda) \rightarrow \infty$ as $\lambda \rightarrow \infty$, but the positive number $c$ can be chosen in such a way that it is independent of $\lambda$ (for example, we may set $c = \pi$). Thus  
	$$\frac{\sum_{n=1}^{N(\lambda)} a(\Dc_n)}{a(\lambda\Pi)} \ge \frac{\pi}{a^*(6) + \frac{2c}{N(\lambda)}}$$
	and hence 
	$$\liminf_{\lambda \to \infty} \frac{\sum_{n=1}^{N(\lambda)} a(\Dc_n)}{a(\lambda\Pi)} \ge \frac{\pi}{a^*(6)} = \frac{\pi}{1 + \frac{\sqrt 3}{2}}.$$
	Therefore, 
	$$\gamma^*(\Dc) \ge \frac{2\pi}{2 + \sqrt{3}}.$$
\end{proof}

Now we produce a sequence covering of the plane by closed unit circular disks that attains this bound. Suppose that for all $j \in \mathbb{N}$, $\Pi_j$ denotes the regular convex dodecagon (12-gon) whose vertices (in polar coordinates, taking the origin as the pole and the $x$-axis as the polar axis) are the elements of the set 
$$\left\{\left(\frac{(1+\sqrt{3})\cdot j}{\sqrt{2}},\frac{\pi \cdot i}{6}\right) : i \in \{0,\ldots,11\}\right\}.$$

\begin{REMARK}\normalfont
	For all $j \in \mathbb{N}$, the length of each side of the convex regular dodecagon $\Pi_j$ as defined above is $j$ and hence $a(\Pi_j) = 3\left(2+\sqrt{3}\right)j^2$.
\end{REMARK}
Now for all $j \in \mathbb{N}$ we place unit circular disks with centers at each vertex of $\Pi_j$ and at any point on a side of $\Pi_j$ such that the point is an integral distance away from the end-vertices of the side. We note that in this way for all $j \in \mathbb{N}$, we place $12j$ unit circular disks along the boundary of $\Pi_j$. For all $j \in \mathbb{N}$, we call the collection of these $12j$ unit disks 
the ``\emph{$j$-th layer" of disks}. We also call the unit disk centered at the origin the ``\emph{$0$-th layer" of disks} (see Figure~\ref{fig-layers}).

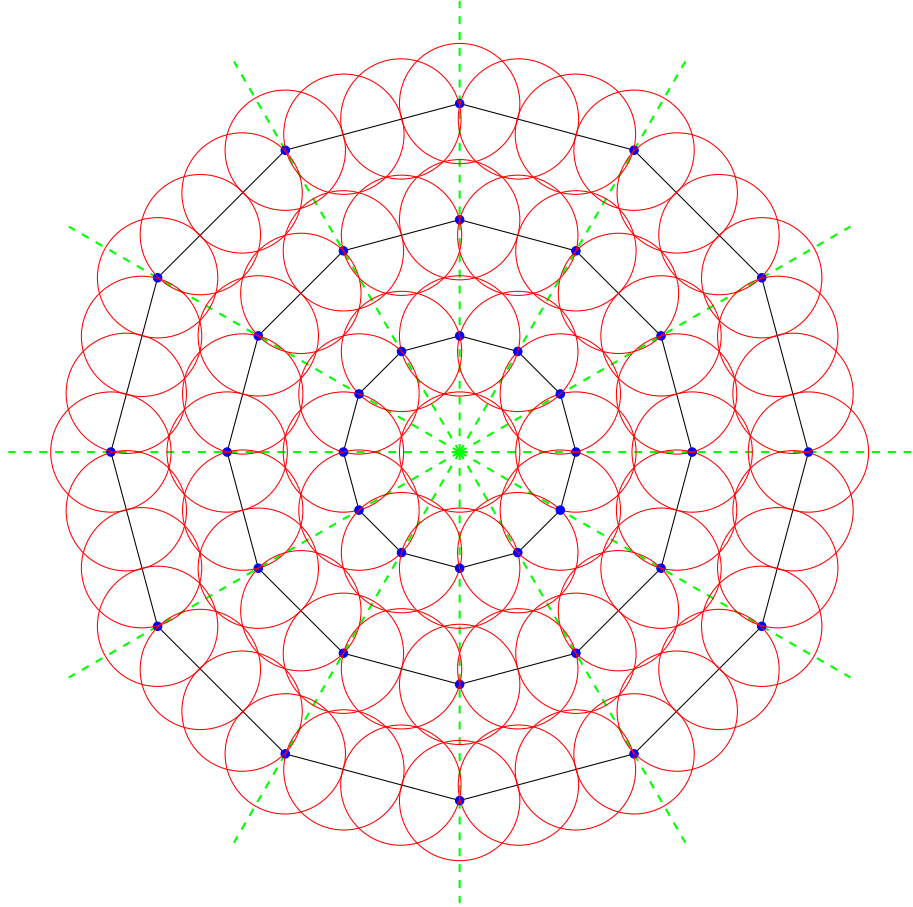
\begin{figure}[!ht]
	\centering
	\begin{tikzpicture}
	\begin{scope}[scale=0.8, transform shape]
	
	\foreach \a in {0, 30, ..., 350}
	\draw [green,dashed,thick] (\a:7.5) (\a:0) -- (\a+30:7.5);
	
	\foreach \a in {0, 30, ..., 350}
	{
		\draw [] (\a:1.932) (\a:1.932) -- (\a+30:1.932);
		\draw [] (\a:3.864) (\a:3.864) -- (\a+30:3.864);
		\draw [] (\a:5.796) (\a:5.796) -- (\a+30:5.796);
	}

	\draw[red] (0,0) circle (1);
	
	\foreach \a in {0, 30, ..., 330 }
	{\draw[red] (\a:1.932) circle (1);
		\draw[red] (\a:3.864) circle (1);
		\draw[red] (\a:5.796) circle (1);
		\node[circle, draw, blue, fill=blue, inner sep=0pt, minimum size=4pt] at (\a:1.932) {};	
		\node[circle, draw, blue, fill=blue, inner sep=0pt, minimum size=4pt] at (\a:3.864) {};	
		\node[circle, draw, blue, fill=blue, inner sep=0pt, minimum size=4pt] at (\a:5.796) {};	
	}
	
	\foreach \a in {15,45, ..., 345 }
	{\draw[red] (\a:3.732) circle (1);}
	
	\foreach \a in {10,40, ..., 350}
	{\draw[red] (\a:5.625) circle (1);}
	\foreach \a in {20,50, ..., 350}
	{\draw[red] (\a:5.625) circle (1);}

	\end{scope}
	\end{tikzpicture}
	
	\caption{Dodecagons $\Pi_1$, $\Pi_2$, and $\Pi_3$ (smaller to bigger) and $0$-th, first, second, and third layer of disks}
	\label{fig-layers}
\end{figure}

We may verify that the collection of all the disks belonging to all the $j$-th layer of disks (including the $0$-th layer of disks) covers $\mathbb {R}^2$. More importantly, for each $j$, the $12j$ disks belonging in the $j$-th layer of disks may be ordered in such a way that from the second disk onward, each disk in the layer has its center on the boundary of the previous disk. We can do that by walking in the clockwise (or anti-clockwise) direction along the boundary of $\Pi_j$ starting with the center of any arbitrary disk in the $j$-th layer and picking up the next disks in order as we approach their centers.

We can also jump from the ``last" disk in the $j$-th layer to the ``first" disk in the $(j+1)$-th layer with only finitely many disks in-between which respect both the ``disk having its center inside the previous disk" and ``no sharp turn beyond $\frac{2\pi}{3}$" criteria. As demonstrated in the Figure~\ref{fig-jump}, three circular disks (all lie 
within the $(j-1)$-th and $j$-th layers) are enough to jump from the last disk of the $(j-1)$-th layer to the first disk of the $j$-th layer for a large enough $j$ (say $j \ge 5$) and also for a smaller $j$, we can jump to the next layer of disks using finitely many extra disks. Thus there is a $c \in \mathbb N$, $c$ being independent of $j$, such that we can jump from the $(j-1)$-th layer to the $j$-th layer using at most $c$ extra disks for all $j \in \mathbb N$.

\begin{figure}[!ht]
	\centering
	\[\begin{tikzpicture}
	\begin{scope}[scale=0.9, transform shape]
	\draw[red, densely dotted, thick] (1,0) circle (1);
	\draw[red, densely dotted, thick] (2,0) circle (1);
	\draw[red, densely dotted, thick] (3,0) circle (1);
	\draw[red, densely dotted, thick] (4,0) circle (1);
	\draw[red, densely dotted, thick] (5,0) circle (1);
	\draw[red, densely dotted, thick] (6,0) circle (1);
	
	\draw[red, densely dotted, thick] (0.5,1.866) circle (1);
	\draw[red, densely dotted, thick] (1.5,1.866) circle (1);
	\draw[red, densely dotted, thick] (2.5,1.866) circle (1);
	\draw[red, densely dotted, thick] (3.5,1.866) circle (1);
	\draw[red, densely dotted, thick] (4.5,1.866) circle (1);
	\draw[red, densely dotted, thick] (5.5,1.866) circle (1);
	\draw[red, densely dotted, thick] (6.5,1.866) circle (1);
	
	\draw[thick] (2.75,0.933) circle (1);
	\draw[thick] (2.375,0.4665) circle (1);
	\draw[thick] (3.125,1.3995) circle (1);

	\draw [dashed,green,thick](2,0)--(3.5,1.866);
	
	\draw [->,green](1,0)--(1.5,0);
	\draw [->,green](3,0)--(3.5,0);
	\draw [->,green](4,0)--(4.5,0);
	\draw [->,green](5,0)--(5.5,0);
	\draw [dotted,->,green,thick](-1,0)--(0.5,0);
	\draw [dotted,->,green,thick](6,0)--(6.5,0);
	
	\draw [green](1.5,0)--(2,0);
	\draw [green](3.5,0)--(4,0);
	\draw [green](4.5,0)--(5,0);
	\draw [green](5.5,0)--(6,0);
	\draw [dotted,green,thick](0.5,0)--(1,0);
	\draw [dotted,green,thick](6.5,0)--(7.5,0);
	\draw [densely dotted,green,thick](1.5,-1.866)--(3,0);
	
	\draw [->,green](3.5,1.866)--(4,1.866);
	\draw [->,green](4.5,1.866)--(5,1.866);
	\draw [->,green](5.5,1.866)--(6,1.866);
	
	\draw [green](4,1.866)--(4.5,1.866);
	\draw [green](5,1.866)--(5.5,1.866);
	\draw [green](6,1.866)--(6.5,1.866);
	\draw [dotted,->,green,thick](6.5,1.866)--(7,1.866);
	\draw [dotted,green,thick](7,1.866)--(8,1.866);
	
	\node at ($(8.2,0)+(0:1 and .6)$) {$j$-th layer};
	\node at ($(8.2,1.866)+(0:1 and .6)$) {$(j+1)$-th layer};

	\end{scope}
	\end{tikzpicture}\]
	\caption{Jumping from $j$-th layer to $(j+1)$-th layer of disks (dotted ones) using three circles (solid) for a large enough $j$}
	\label{fig-jump}
\end{figure}
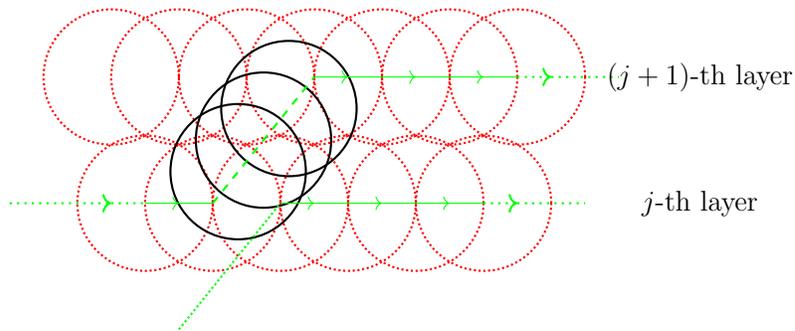

Combining all the appropriate orderings among the disks in the $j$-th layer and orderings of the disks required to jump from the $j$-th layer to the $j+1$-th layer in a suitable way we get an infinite sequence of unit circular disks on $\mathbb {R}^2$, say $(\Dc_n)_{n \in \mathbb N}$ starting with the unit circular disk centered at the origin. We observe that the family of disks $\mathscr S = \{\Dc_n : n \in \mathbb N\}$ covers $\mathbb {R}^2$. Also this sequence of disks $(\Dc_n)$ has the property that the second disk onward, each disk has its center in the previous disk and the smaller angle formed at the center by the line segments joining it with the centers of the previous and the next disks is at least $\frac{2\pi}{3}$.

We assert that the (lower) density of this sequence of unit circular disks is at least $\frac{2\pi}{2 + \sqrt{3}}$ even though this sequence contains many redundant disks.

We note that the collection of disks belonging in the $j$-th layer for all $j \in \{0,\dots,k\}$ covers the dodecagon $\Pi_k$ and (the interior of) $\Pi_k$ doesn't intersect with any of the disks belonging in the $j$-th layer for all $j > k$. Also for a large enough $k$, $\Pi_k$ doesn't intersect with any of the disks required to jump from the $j$-th layer to the $j+1$-th layer for all $j > k$ either. Thus the number of disks that intersect the dodecagon $\Pi_k$, say $N_k$, is at most the sum of the total number of disks in the $0$-th to $k$-th layers and total number of disks required to jump from the $j$-th layer to the next layer with $j \le k$, i.e. 

\begin{align*}
N_k &\le 1+ \sum_{j=1}^{k} 12j + c(k+1) \\
&= 1 + 6k(k+1) + c (k+1) \\
&= 6k^2 + (c+6)k + c+1.
\end{align*}
Therefore, the (lower) density of the covering $\mathscr S$ is less than or equal to  

\begin{align*}
\lim\limits_{k \to \infty} \frac{N_k\cdot \pi}{a(\Pi_k)} 
&\le \lim\limits_{k \to \infty} \frac{\left(6k^2 + (c+6)k + c+1\right)\cdot \pi}{3\left(2+\sqrt{3}\right)k^2}\\
&= \frac{6 \pi}{3\left(2+\sqrt{3}\right)}\\
&=\frac{2\pi}{2 + \sqrt{3}}.
\end{align*}

\section{Covering the plane by disks each containing centers of at least two other disks}
In this section we prove Theorem~\ref{thm1.7}. Let $\{\Dc_n : n \in \mathbb N\}$ be a covering of $\mathbb {R}^2$ by unit circular disks such that  the centers of the disks form a lattice $\Lambda$ in $\mathbb {R}^2$. Let $\{v_1, v_2\}$ be a set of generators of $\Lambda$, i.e. $\Lambda = \{\alpha v_1+\beta v_2 : (\alpha,\beta)\in 
{\mathbb Z}^2\}$. Thus every lattice point $p$ can be represented as a pair $(a,b)\in 
{\mathbb Z}^2$, where $p=av_1+bv_2$. The ``fundamental parallelogram" corresponding to the lattice $\Lambda$ is given by $$\{xv_1+yv_2 : (x,y)\in [0,1]\times [0,1]\}.$$ 
Area of the fundamental parallelogram is equal to the ``\emph{determinant of the lattice}" ($\mathrm{det} (\Lambda)$), i.e.  
$$\mathrm{det} (\Lambda) = |\mathrm{det} [v_1,v_2]|.$$
Now if we generate a Voronoi diagram by setting the lattice points belonging to $\Lambda$ (i.e. centers of the dics) as the generators, the Voronoi cells corresponding to the lattice points are congruent to each other and they tile the plane.
We can also show that area of each Voronoi cell equals $\mathrm{det} (\Lambda)$.

From the definition of covering density and its relation with the area of the Voronoi cells as discussed in the previous section, it follows that the (lower) density of the covering corresponding to the lattice $\Lambda$ 
equals 

$$\frac{a(\Dc)}{\textrm{area of Voronoi cells corresponding to the lattice points}} = \frac {\pi}{\mathrm {det} (\Lambda)}.$$

Now we need to maximize $\mathrm {det} (\Lambda)$ with the constraint that  
each disk contains the centers of at least two other disks. 

\begin{PROPOSITION}
	If the lattice $\Lambda$ corresponds to a lattice covering of the plane by unit circular disks with the constraint that each disk contains the centers of at least two other disks, then $\mathrm {det} (\Lambda) \le 1+ \frac{\sqrt{3}}{2}$.
\end{PROPOSITION}
\begin{proof}
Let $v_1$ be the position vector of the center of a disk closest to (but not centered at) the origin. From the constraint it follows that length of $v_1$ is at most 1. Since the centers form the  two dimensional lattice $\Lambda$, there is a center (of a disk) with the position vector $v_2$ such that $v_1$ and $v_2$ generate $\Lambda$. 

Let, without loss of generality, $v_1=(2\alpha,0)$, where $\alpha \in \left(0, \frac 1 2\right]$. Let us consider the 
point $P =\left(\alpha, \sqrt{1-\alpha^2}+ \epsilon \right)$ where $\epsilon > 0$. This 
point is not covered by any disk whose center lies on the lattice 
points on the $x$-axis. Let the position vector of the center of a disk 
that covers $P$ be $w=(\beta_1,\beta_2)$ (see Figure~\ref{fig-inout}). 

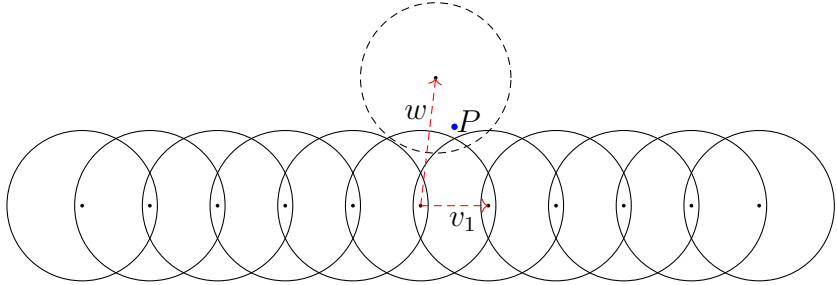
\begin{figure}[!ht]
	\centering
	\begin{tikzpicture}
	\begin{scope}[scale=1, transform shape]
	
	\foreach \a in {0, 0.9, ...,9}
	\draw (\a,0) circle (1);

	\foreach \a in {0, 0.9, ...,9}
	\node[circle, draw, fill=black, inner sep=0pt, minimum size=1pt] at (\a,0) {};
	
	\node[xshift=4.5cm, circle, draw, blue, fill=blue, inner sep=0pt, minimum size=2pt] at (0.45,1.05) {};
	
	\draw [densely dashed] (4.7,1.7) circle (1);
	\node[xshift=4.7cm, circle, draw, fill=black, inner sep=0pt, minimum size=1pt] at (0,1.7) {};
	
	\node at ($(4.85,0)+(-45:.3 and .3)$) {$v_1$};
	\node at ($(4.7,1)+(150:.3 and .5)$) {$w$};
	\node at ($(4.85,1)+(15:.3 and .5)$) {$P$};

	\draw [xshift=4.5cm,red,densely dashed,->] (0,0)--(0.9,0);
	\draw [xshift=4.5cm,red,densely dashed,->] (0,0)--(0.2,1.7);
	
	\end{scope}
	\end{tikzpicture}
	
	\caption{Point $P=\left(\alpha, \sqrt{1-\alpha^2}+ \epsilon \right)$ and the disk centered at $(\beta_1,\beta_2) = w$ containing $P$}
	\label{fig-inout}
\end{figure}

We want to find an upper-bound on the area of the parallelogram, say $\Ply$, formed by the vectors $v_1$ and $w$. We note that $a(\Ply) = 2\alpha\cdot |\beta_2|$. Since the unit disk centered at $(\beta_1,\beta_2)$ contains the point $P$, we get $|\beta_2| \le \left(\sqrt {1-\alpha^2}+\epsilon\right)+1$. This implies $a(\Ply) \le 2\alpha \cdot \left(\sqrt {1-\alpha^2}+1+\epsilon\right)$. 
The function $$f(\alpha) = 2\alpha\cdot \left(\sqrt {1-\alpha^2}+1+\epsilon\right)$$ is increasing on $\left(0,\frac 1 2\right]$ and hence attains its maximum at $\alpha =\frac 1 2$.
Also, since $\epsilon$ is arbitrary, we get $a(\Ply) \le 1+ \frac{\sqrt{3}}{2}$.

Since the vectors $v_1$ and $v_2$ generate the lattice $\Lambda$, $a(\Ply)$ is a positive integral multiple of the area of the parallelogram formed by $v_1$ and $v_2$, viz. $\mathrm {det} (\Lambda)$. Thus $\mathrm {det} (\Lambda) \le a(\Ply) \le 1+ \frac{\sqrt{3}}{2}$.
\end{proof}

The lattice generated by the vectors $v_1=(1,0)$ and $v_2= \left(\frac 1 2, 1+ \frac{\sqrt{3}}{2}\right)$ as shown in Figure~\ref{fig-inout-lattice} corresponds to a lattice covering of the plane by unit circular disks such that each disk contains the centers of two other disks and the lattice determinant equals $1+ \frac{\sqrt{3}}{2}$.
This shows it is an optimal lattice covering with the given constraint and has the (lower) covering density 
$\frac {\pi}{1+ \frac{\sqrt{3}}{2}} = \frac{2\pi}{2 + \sqrt{3}}$. This proves Theorem~\ref{thm1.7}.

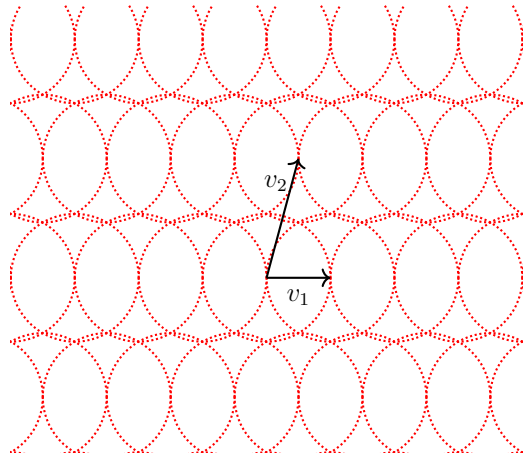
\begin{figure}[!ht]
	\centering
	\begin{tikzpicture}
	\begin{scope}[scale=0.85, transform shape]
	\clip(2,1) rectangle (10,8);
	\foreach \a in {0, 1, ...,12}
	\draw [red,densely dotted, thick] (\a,0) circle (1);
	
	\foreach \a in {0, 1, ...,12}
	\draw [xshift=0.5cm,yshift=1.866cm,red,densely dotted, thick] (\a,0) circle (1);
	
	\foreach \a in {0, 1, ...,12}
	\draw [yshift=3.732cm, red,densely dotted, thick] (\a,0) circle (1);
	
	\foreach \a in {0, 1, ...,12}
	\draw [xshift=0.5cm,yshift=5.598cm, red,densely dotted, thick] (\a,0) circle (1);
	
	\foreach \a in {0, 1, ...,12}
	\draw [yshift=7.464cm, red,densely dotted, thick] (\a,0) circle (1);
	
	\draw [yshift=3.732cm,xshift=6cm,->, thick] (0,0)--(1,0);
	\draw [yshift=3.732cm,xshift=6cm,->, thick] (0,0)--(0.5,1.866);

	\node at ($(6.5,3.732)+(-90:.3 and .3)$) {$v_1$};
	\node at ($(6.2,4.72)+(100:.3 and .5)$) {$v_2$};
	\end{scope}
	\end{tikzpicture}
	
	\caption{Covering corresponding to the lattice generated by the vectors $v_1 = (1,0)$ and $v_2 = \left(\frac 1 2, 1+ \frac{\sqrt{3}}{2}\right)$}
	\label{fig-inout-lattice}
\end{figure}

\section{Conclusions}
Problem~\ref{prob-1} can be generalized as follows. Let $(\Dc_n)_{n \in \mathbb N}$ be a sequence of unit circular disks that covers $\mathbb {R}^2$ in such a way that for all $n \ge 2$, distance between the centers of the disks $\Dc_{n-1}$ and $\Dc_n$ is at most a positive real number $\rho$. 
What is the optimal lower density of covering the plane by a sequence of unit circular disks with the aforementioned constraint? Can the optimal lower density be achieved by a suitable sequence covering?

If $\rho$ is large enough (for example when $\rho$ is close to 2), then the problem can be reduced to the problem of covering the plane with unit circular disks without any constraints. We considered the case $\rho = 1$ in Section~2 and established a bound on the optimal lower density with a restriction on the ``turning angles". We can use the techniques used in Section~2 to find out a bound for lower covering density for any $\rho$ with an analogous restriction on the turning angles depending on $\rho$. In Section~2 we have also constructed a sequence of disks which attains the established lower bound. The same idea may be useful for constructing a sequence of disks covering the plane with the optimal density for a given $\rho$ and the corresponding restriction on the turning angles.

We also observe that for a sequence of disks covering the plane admitting the restriction in Problem~\ref{prob-1}, second disk onward, \emph{new} area (i.e. not already covered by the previous disks) that can be covered by each disk is at most the area of the shaded crescent, say $\C$ in Figure~\ref{fig-crescent}. 

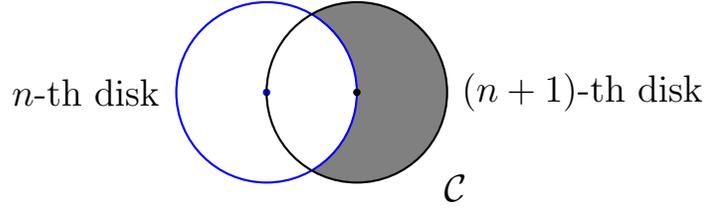
\begin{figure}[!ht]
	\centering
	\begin{tikzpicture}
	\begin{scope}[scale=1.2, transform shape]
	
	\path [ fill=gray]
	(0.5,-0.866) arc [radius=1, start angle=-120, end angle=120]
	--
	(0.5,0.866) arc [radius=1, start angle=60, end angle=-60] 
	;

	\draw [blue, thick] (0,0) circle (1);
	\node[circle, draw, blue, fill=blue, inner sep=0pt, minimum size=2pt] at (0,0) {};	
	
	\draw [thick](1,0) circle (1);
	\node[circle, draw, , fill=black, inner sep=0pt, minimum size=2pt] at (1,0) {};

	\node  at ($(-1,0)+(180:1 and 2.3)$) {$n$-th disk};
	\node  at ($(2,0)+(0:1.5 and 2.5)$) {$(n+1)$-th disk};
	\node  at ($(2,0)+(-25:0.1 and 2.5)$) {$\C$};
	\end{scope}
	\end{tikzpicture}
	
	\caption{Shaded crescent $\C$, maximum ``new" area that can be covered by the $(n+1)$-th disk}
	\label{fig-crescent}
\end{figure}

We note that $a(\C) = \sin \frac{2\pi}{3} + \frac{\pi}{3}$. Therefore, if $\gamma'(\Dc)$ is the optimal (lower) density of covering the plane by a sequence of unit circular disks in such a way that each disk (excluding the first one) contains the center of the previous one, then 
$$\gamma'(\Dc) \ge \frac{a(\Dc)}{a(\C)} = \frac{\pi}{\sin \frac{2\pi}{3} + \frac{\pi}{3}} \approx 1.64204.$$
We also note that the lower bound obtained in this way is a very crude one; this bound is not achievable here as it can only be achieved by a covering which is ``almost like a tiling" (by the crescent shape). For example an analogical crude lower bound on the optimal (lower) density of covering by congruent circular disks without any constraints is $1$ whereas the actual achievable lower bound as shown by Kershner is $\frac{2\pi}{\sqrt{27}}  \approx 1.20920$. 

From this observation and the results in Section~2 and Section~3 it seems that the restriction on the turning angles is only due to the limitations of the proof-techniques used. We observe that if we get rid of the angle restriction, i.e. if, for some $n \in \mathbb{N}$, $\angle C_{n-1}C_nC_{n+1} < \frac{2\pi}{3}$ (we refer to Figure~\ref{fig-conj} for an example), then the corresponding $\nu_n$-gon $\Ply_n$ constructed as in the proof of Lemma~\ref{lemma-poly} would be contained in a convex, compact set larger than $\M'_\theta$ for $\frac{2\pi}{3} \le \theta \le \pi$ or in particular larger than $\M$ (for example, $\Ply_n$ would be contained in the convex, compact set bounded by solid lines in Figure~\ref{fig-conj}). 

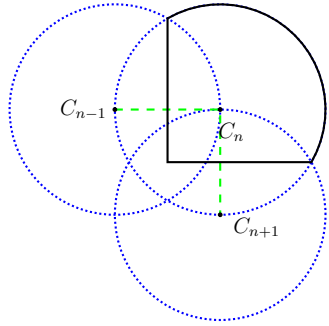
\begin{figure}[!ht]
	\centering
	\[\begin{tikzpicture}
		\begin{scope}[scale=0.7, transform shape]
	\draw[densely dotted, thick, blue] (-2,0) circle (2);
	\draw[densely dotted, thick, blue] (0,0) circle (2);
	\draw[densely dotted, thick, blue] (0,-2) circle (2);
	
	\draw [dashed,green,thick](-2,0)--(0,0)--(0,-2);
	
	\node[circle, draw, fill=black, inner sep=0pt, minimum size=2pt] at (0,0) {};
	\node[circle, draw, fill=black, inner sep=0pt, minimum size=2pt] at (-2,0) {};
	\node[circle, draw, fill=black, inner sep=0pt, minimum size=2pt] at (0,-2) {};
	
	\draw [thick,domain=-30:120] plot ({2*cos(\x)}, {2*sin(\x)});
	\draw [thick](-1,1.75)--(-1,-1)--(1.75,-1);

	\node at ($(-2,0)+(180:.6 and .6)$) {$C_{n-1}$};
	\node at ($(0,0)+(-45:.3 and .6)$) {$C_n$};
	\node at ($(0,-2)+(-30:.8 and .5)$) {$C_{n+1}$};

	\end{scope}
	\end{tikzpicture}\]
	\caption{Example of three disks in a sequence with $\angle C_{n-1}C_nC_{n+1} < \frac{2\pi}{3}$}
	\label{fig-conj}
\end{figure}

In fact when $\nu_n$ is large enough (for example bigger than 6), in specific cases the inequality $a(\Ply_n) \le a(\Ply^*_n)$ (conclusion of Lemma~\ref{lemma-bigger-poly}) will no longer hold. But we also observe that if $\angle C_{n-1}C_nC_{n+1} < \frac{2\pi}{3}$, then it induces a larger intersection of $\Dc_{n-1}$ and $\Dc_{n+1}$ which would potentially reduce these two disks' ``contributions" in the covering. Also for the $\nu_n$-gon $\Ply_n$, larger $\nu_n$ is, more congested the neighboring disks of $\Dc_n$ are. This again would potentially reduce the contributions of the neighboring disks of $\Dc_n$. Thus it seems we would never gain any significant advantage globally by allowing sharper turns.

We may also note that this proof-technique provides a bound on optimal density of covering the plane by congruent copies of a convex, compact set in $\mathbb{R}^2$ with a non-empty interior with the restriction that no two sets \emph{cross each other} (Corollary~\ref{FT-bound}). But it is not necessary for a covering to be ``crossing-free" in order to be an optimal covering.
For a ``fat" ellipse (i.e. sufficiently close to a circle) the lower bound on the density of covering the plane with congruent copies of the ellipse obtained from L. Fejes T\'{o}th's theorem (Corollary~\ref{FT-bound}) holds true \cite{Heppes2003}, and thus the optimal lattice covering density is the most economical. But it is not known whether an optimal lattice covering is an optimal covering for an elongated ellipse. Moreover, G.~Fejes T\'{o}th and W.~Kuperberg pointed out in \cite{tóth_kuperberg_1995} that it is very likely that in higher-dimensional spaces the most economical coverings with sufficiently long ellipsoids is \emph{not} crossing-free.
All these motivate us to propose Conjecture~\ref{conj}.

\section*{Acknowledgements}
We are grateful to the anonymous referees who reviewed the journal version of the article for their valuable remarks and suggestions which substantially improved the presentation of this article. We are also thankful to Mainak Ghosh of School of Mathematics, Tata Institute of Fundamental Research for kindly proofreading an initial draft of the manuscript and providing feedback.

 \bibliographystyle{plain} 

\end{document}